\documentclass[leqno,12pt]{amsart}
\usepackage{amsmath,amsthm, amsfonts,amssymb}
\usepackage{enumerate}
\usepackage{hyperref}
\usepackage[T1]{fontenc}
\newtheorem{thm}{Theorem}[section]

\newtheorem{cor}[thm]{Corollary}
\newtheorem{pro}[thm]{Proposition}

\theoremstyle{definition}

\theoremstyle{remark}
\newtheorem*{remark}{Remark}
\newtheorem*{remark1}{Remark 1}
\newtheorem*{remark2}{Remark 2}
\newtheorem*{remark3}{Remark 3}

\newcommand{\la}{\lambda}

\newcommand{\mF}{\mathcal{F}}
\newcommand{\mL}{\mathcal{L}}

\newcommand{\bnH}{\tilde{H}}

\newcommand{\Hla}{\mathcal{H}_{\alpha}}
\newcommand{\pst}{\varphi_p^*}

\DeclareMathOperator{\Arg}{Arg}
\DeclareMathOperator{\Dom}{Dom}
\DeclareMathOperator{\Ran}{Ran}
\DeclareMathOperator{\supp}{supp}

\title[Riesz transforms and $H^{\infty}$ joint functional calculus]{Dimension free $L^p$ estimates for Riesz transforms via an $H^{\infty}$ joint functional calculus}
\author{B\l a\.{z}ej Wr\'obel}
\address{Instytut Matematyczny, Uniwersytet Wroc\l awski, pl. Grunwaldzki 2/4, 50-384 Wroc\l aw, Poland \newline \&
Scuola Normale Superiore, Piazza dei Cavalieri 7, 56126 Pisa,   Italy
}
\email{\ \newline
blazej.wrobel@math.uni.wroc.pl}
\thanks{I thank Prof.\ Fulvio Ricci, for a certain conversation that inspired the creation of this article.}
 
\begin{document}

 \begin{abstract}

  By using an $H^{\infty}$ joint functional calculus for strongly commuting operators, we derive a scheme to deduce the $L^p$ boundedness of certain $d$-dimensional Riesz transforms from the $L^p$ boundedness of appropriate one-dimensional Riesz transforms. Moreover, the $L^p$ bounds we obtain are independent of the dimension. The scheme is applied to Riesz transforms connected with orthogonal expansions and discrete Riesz transforms on products of groups with polynomial growth.
 \end{abstract}
  \maketitle
 \numberwithin{equation}{section}
 \section{Introduction}

        In \cite{SteinRiesz} Stein proved dimension free $L^p$ estimates for the classical Riesz transforms on $\mathbb{R}^d.$ Since then many analogous results were obtained for Riesz transforms defined in various settings.

       Several authors investigated dimension free $L^p$ estimates for Riesz transforms connected with (discrete) orthogonal expansions. In the Hermite polynomial case (Ornstein-Uhlenbeck-Riesz transforms) this was studied by Meyer \cite{Mey1}, Pisier \cite{Pis1}, and Gutierrez \cite{Gut1}. Dimension free bounds for Hermite-Riesz transforms (the setting of Hermite function expansions) can be deduced, by means of transference, from the paper of Couhlon, M\"uller, and Zienkiewicz \cite{CMZ} (see also \cite{HRST} and \cite{LP2} for different proofs). The Jacobi polynomial setting was treated by Nowak and Sj\"ogren \cite{NSj2}, who obtained dimension and parameter free estimates for the considered Riesz transforms. In the Laguerre polynomial case the topic was initiated by Guttierrez, Incognito and Torrea \cite{GIT} (half-integer multi-indices), and completed by Nowak \cite{No1} (continuous range of multi-indices). In both \cite{GIT} and \cite{No1} the authors proved dimension and parameter free bounds for the considered Riesz transforms. Dimension and parameter free estimates for Riesz transforms in the setting of Laguerre function expansions of Hermite type were obtained by Stempak and the author \cite{StWr}. Additionally, Dragi\v{c}evi\'{c} and Volberg in \cite{DragVol0} and  \cite{DragVol1} (see also \cite{DragVol2}), proved presumably optimal in $p$ dimension free $L^p$ estimates for the Ornstein-Uhlenbeck-Riesz transforms, and Hermite-Riesz transforms, respectively.

       The phenomenon of dimension free $L^p$ bounds for Riesz transforms was also studied in other contexts. In the above-mentioned article \cite{CMZ}, Couhlon, M\"uller, and Zienkiewicz, showed dimension free estimates for Riesz transforms associated to the sub-Laplacian on the Heisenberg group. In two papers \cite{Lu_Piqu2} and \cite{Lu_Piqu1}, Lust-Piquard proved dimension free bounds for Riesz transforms on $l^p(\{-1,1\}^d),$ and, more generaly, for discrete Riesz transforms on products of abelian groups. Li \cite{Li1}, generalizing an earlier result of Bakry \cite{Bakr1}, proved a dimension free estimate for Riesz transforms on complete Riemannian manifolds, under the assumption that the Bakry-Emery Ricci curvature is bounded from below. In \cite{UrZien} Urban and Zienkiewicz investigated dimension free bounds for Riesz transforms of some Schr\"odinger operators.

        We introduce a setting which formally encompasses all Riesz transforms built in a 'product' manner. In particular it covers all the Riesz transforms connected with (discrete) orthogonal expansions mentioned above, as well as those associated with (continuous) Bessel expansions (which were studied in \cite{Bet1}). This setting also includes the Riesz transforms studied in \cite{Lu_Piqu1} and \cite{Lu_Piqu2}.

        Assume that $L=(L_1,\ldots,L_d)$ is a system of non-negative self-adjoint operators such that each $L_r$ acts on $L^2(X_r,\nu_r),$ $r=1,\ldots,d.$ By $r$-th first order multi-dimensional Riesz transform associated to the system $L$ we mean \begin{equation}\label{eq:Rieszdef} R_r=\delta_r(L_1+\ldots + L_d)^{-1/2}.\end{equation}
        Here $\delta_r$ is a certain operator acting on a dense subspace of $L^2(X_r,\nu_r)$ (hence, by a tensorization argument also on a dense subspace of $L^2(\prod_{r=1}^d X_r, \bigotimes_{r=1}^d \nu_r)$).
        To be precise, if $0$ is an eigenvalue of $L,$ then the definition of $R_r$ needs to be slightly modified; this is properly explained in all the needed cases. Often $\delta_r$ and $L_r$ are related by $L_r=\delta_r^*\delta_r+a_r,$ with $a_r\geq 0;$ in particular this is the case in Theorems \ref{thm:RieszOrt} and \ref{thm:Rieszdiscgen}.

        In Section \ref{sec:ort} we show how \eqref{eq:Rieszdef} can be formalized for Riesz transforms connected with (discrete and continuous) orthogonal expansions. Then, in Section \ref{sec:disc} we formalize \eqref{eq:Rieszdef} for Riesz transforms on products of discrete groups having polynomial volume growth. For the time being let us only note that the classical Riesz transforms are indeed of the form \eqref{eq:Rieszdef} with $\delta_r=\partial_r,$ $L_r=-\partial_r^2,$ and $(X_r,\nu_r)=(\mathbb{R},dx_r),$  $r=1,\ldots,d.$

        The main goal of this paper is to establish a connection between an $H^{\infty}$ joint functional calculus (for a pair of strongly commuting operators) and (dimension free) $L^p$ boundedness of multi-dimensional Riesz transforms given by \eqref{eq:Rieszdef}. In Corollary \ref{cor:Scheme} we present a scheme allowing to deduce dimension free $L^p$ bounds for these Riesz transforms. Here a crucial tool is the functional calculus from Proposition \ref{cor:genMarHinfdimindep}.

         Corollary \ref{cor:Scheme} is then applied to Riesz transforms connected with orthogonal expansions (see Theorem \ref{thm:RieszOrt}), and to Riesz transforms on products of discrete groups having polynomial volume growth (see Theorem \ref{thm:Rieszdiscgen}). For the first of these Riesz transforms our method gives several new results. Namely, we obtain dimension free $L^p$ bounds for multi-dimensional Riesz transforms in the settings of: Laguerre expansions of convolution type, Jacobi function expansions, Fourier-Bessel expansions, and Bessel expansions (the Hankel transform setting). Additionally, in the cases of Laguerre polynomial expansions and Laguerre function expansions of Hermite type we enlarge the range of admitted parameters. For the Riesz transforms on products of discrete groups, using Theorem \ref{thm:Rieszdiscgen} we are able to improve \cite[Theorem 2.8]{Lu_Piqu1}, see Corollary \ref{cor:RieszgenLu_Piqu}.

        Let us remark that the methods for proving dimension free bounds we present here seem quite general, whereas the anterior methods were rather tailored to each of the specific cases. It also seems that because of that generality we can not hope to obtain dimension free bounds for vectors of Riesz transforms. Indeed, the counterexample in \cite[Proposition 2.9]{Lu_Piqu1} shows that, even for the discrete Riesz transforms on $\mathbb{Z}^d,$ there are no dimension free bounds for the vector of Riesz transforms on $l^p(\mathbb{Z}^d),$ for $1<p<2.$

        A large part of the material presented in this paper was included in the PhD thesis of the author \cite{PhD}.

\section{Preeliminaries}
\label{sec:set}
 Consider a system $L=(L_1,\ldots,L_d)$ of non-negative self-adjoint operators on $L^2(X,\nu),$ for some $\sigma$-finite measure space $(X,\nu).$ We assume that the operators $L_r$ commute strongly, i.e.\ that their spectral projections $E_{L_r},$ $r=1,\ldots,d,$ commute pairwise. In that case, there exists the joint spectral measure $E$ associated with $L,$ determined uniquely by the condition
        \begin{equation*}
        L_r=\int_{[0,\infty)}\la_r dE_{L_r}(\la_r)=\int_{[0,\infty)^d} \la_r dE(\la),
        \end{equation*}
        see \cite[Theorem 4.10 and Theorems 5.21, 5.23]{schmu:dgen}.
        Consequently, for a Borel measurable function $m$ on $[0,\infty)^d,$ the multivariate spectral theorem allows us to define
         \begin{equation}
         \label{m(L)def}
        m(L)=m(L_1,\ldots,L_d)=\int_{[0,\infty)^d} m(\la) dE(\la)
        \end{equation}
        on the domain
         \begin{equation}\label{m(L)dom} \Dom(m(L))=\bigg\{f\in L^2(X,\nu)\colon \int_{[0,\infty]^d}|m(\la)|^2dE_{f,f}(\la)<\infty\bigg\}.\end{equation}
         Here $E_{f,f}$ is the complex measure defined by $E_{f,f}(\cdot)=\langle E(\cdot)f,f\rangle_{L^2(X,\nu)}.$

        The crucial assumption we make is the $L^p(X,\nu)$ contractivity of the heat semigroups $e^{-tL_r},$ $r=1,\ldots,d.$ More precisely, we impose that, for all $1\leq p\leq \infty$ and $t_r>0,$
\begin{equation*}
\tag{CTR}
\label{contra}
\|e^{-t_rL_r}f\|_{L^p(X,\nu)}\leq \|f\|_{L^p(X,\nu)},\qquad f\in L^p(X,\nu)\cap L^2(X,\nu).
\end{equation*}

 For technical reasons we also assume the atomlessness condition
\begin{equation*}\tag{ATL} \label{noatomatzero}E_{L_r}(\{0\})=0,\qquad r=1,\ldots,d,\end{equation*} where  $E_{L_r}$ denotes the spectral measure of $L_r.$ Note that, under \eqref{noatomatzero}, the definition of $m(L)$ may be rewritten as $m(L)=\int_{(0,\infty)^d}m(\la)dE(\la).$

Later on we often work in a product setting, when $(X,\nu)=(\Pi_{r=1}^dX_r,\bigotimes_{r=1}^d\nu_r).$ In that case, for a self-adjoint or bounded operator $T$ on $L^2(X_r,\nu_r)$ we define
\begin{equation}
\label{eq:tensnot}
T\otimes I_{(r)}=I_{L^2(X_1,\nu_1)}\otimes \cdots \otimes I_{L^2(X_{r-1},\nu_{r-1})}\otimes T\otimes I_{L^2(X_{r+1},\nu_{r+1})} \otimes \cdots \otimes I_{L^2(X_d,\nu_d)}.
\end{equation} If $T$ is self-adjoint, then the operators $T\otimes I_{(r)}$ can be regarded as non-negative self-adjoint and strongly commuting operators on $L^2(X,\nu),$ see \cite[Theorem 7.23]{schmu:dgen} and \cite[Proposition A.2.2]{PhD}. Moreover, if $T$ is bounded on $L^p(X_r,\nu_r),$ for some $1\leq p<\infty,$ then, by Fubini's theorem $T\otimes I_{(r)}$ is a bounded operator on $L^p(X,\nu)$ and
        \begin{equation}
        \label{eq:tensnormeq}
        \|T\|_{L^p(X_r,\nu_r)\to L^p(X_r,\nu_r)}=\|T\otimes I_{(r)}\|_{L^p(X,\nu)\to L^p(X,\nu)}.
        \end{equation}

In the applications we give, i.e.\ in Theorems \ref{thm:RieszOrt} and \ref{thm:Rieszdiscgen}, the operators $L_r$ act on different variables, that is, each $L_r$ is defined on $\Dom(L_r)\subset L^2(X_r,\nu_r),$ $r=1,\ldots,d.$ Then, it can be proved that $E_{L_r}\otimes I_{(r)}$ is the spectral measure of the operator $L_r\otimes I_{(r)},$ $r=1,\ldots,d.$ From this observation it follows that, if, for some $r=1,\ldots,d,$ the operator $L_r$ satisfies the atomlessness condition \eqref{noatomatzero}, then the same is true for $L_r\otimes I_{(r)}.$

Additionally, for a bounded function $m_r:(0,\infty)\to \mathbb{C},$ we have $m_r(L_r)\otimes I_{(r)}=m_r(L_r\otimes I_{(r)}).$ Thus, using \eqref{eq:tensnormeq}, it is not hard to deduce that, if the operator $L_r$ satisfies the contractivity condition \eqref{contra} (with respect to $L^p(X_r,\nu_r)$), then $L_r\otimes I_{(r)}$ satisfies \eqref{contra} as well (with respect to $L^p(X,\nu)$).

Throughout the paper the following notation is used. The symbols $\mathbb{N}_0$ and $\mathbb{N}$ stand for the sets of non-negative and positive integers, respectively. By $\mathbb{N}_0^d$ we mean $(\mathbb{N}_0)^d.$

For a pair of angles $\varphi=(\varphi_1,\varphi_2)\in (0,\pi/2]^2,$ we denote by ${\bf S}_{\varphi}$ the polysector (contained in the product of the right complex half-planes) $${\bf S}_{\varphi}=\{(z_1,z_2)\in \mathbb{C}^2\colon z_r\neq0,\quad |\Arg(z_r)|<\varphi_r,\quad r=1,2\}.$$ In particular, ${\bf S}_{\pi/2,\pi/2}$ means the product of the right half-planes. If $U$ is an open subset of $\mathbb{C}^2,$ the symbol $H^{\infty}(U)$ stands for the vector space of bounded functions on $U,$ which are holomorphic in several variables. The space $H^{\infty}(U)$ is equipped with the supremum norm.

 We use the variable constant convention, i.e.\ constants (such as $C,$ $C_p$ or $C(p),$ etc.) may vary from one occurrence to another. In most cases we shall however keep track of the parameters on which the constant depends, (e.g.\ $C$ denotes a universal constant, while $C_p$ and $C(p)$ denote constants which may also depend on $p$).

 Let $B_1,B_2$ be Banach spaces and let $F$ be a dense linear subspace of $B_1.$ We say that a linear operator $T\colon F\to B_2$ is bounded, if it has a (unique) bounded extension to $B_1.$

\section{The key theorem and its consequences}
\label{sec:Riesz}
      The crucial result for the paper is the following theorem.

        \begin{thm}
         \label{thm:genRiesz} Let $L=(L_1,\ldots,L_d)$ be a general system of non-negative self-adjoint strongly commuting operators on $L^2(X,\nu)$ satisfying \eqref{contra} and \eqref{noatomatzero}. Then, for each $r=1,\ldots,d,$ and $\sigma>0,$ the operator $L_r^{\sigma}(\sum_{r=1}^d L_r)^{-\sigma}$ is bounded on all $L^p(X,\nu)$ spaces, $1<p<\infty.$ Moreover,
         \begin{equation}
         \label{eq:thm:genRiesz}
         \max_{r=1,...,d}\|L_r^{\sigma}(L_1+\cdots+L_d)^{-\sigma}\|_{L^p(X,\nu)\to L^p(X,\nu)}\leq C_{p,\sigma},
         \end{equation}
         where the constant $C_{p,\sigma}$ is independent of ('the dimension') $d.$
         \end{thm}
         We deduce Theorem \ref{thm:genRiesz} from the following $H^{\infty}$ joint functional calculus for a pair of strongly commuting operators. In what follows $\pst=\arcsin|2/p-1|.$
         \begin{pro}
         \label{cor:genMarHinfdimindep}
         Let $(T,S)$ be a pair of non-negative self-adjoint strongly commuting operators on $L^2(X,\nu)$ satisfying \eqref{contra} and \eqref{noatomatzero}. Fix $1<p<\infty$ and let $m$ be a bounded holomorphic function on ${\bf S}_{\varphi},$ for some $\varphi=(\varphi_1,\varphi_2),$ with $\varphi_r>\pst,$ $r=1,2.$ Then the multiplier operator $m(T,S)$ is bounded on $L^p(X,\nu)$ and
         \begin{equation}
         \label{eq:dimindepineq}
         \|m(T,S)\|_{L^p(X,\nu)\to L^p(X,\nu)}\leq C_p\|m\|_{H^{\infty}({\bf S}_{\varphi})},
         \end{equation}
         where the constant $C_p$ depends only on $p$ and not on the operators $T$ and $S.$
         \end{pro}
         \begin{proof}[Proof (sketch)]
         The proposition follows from an application of a variant of \cite[Theorem 5.4]{AlFrMc} due to Albrecht, Franks and McIntosh. By a recent result of Carbonaro and Dragi\v{c}evi\'{c} \cite[Theorem 1]{Carb-Drag} each of the operators $T$ and $S$ has a bounded $H^{\infty}$ functional calculus in $S_{\varphi_r},$ $r=1,2.$ A more detailed and slightly different justification of the proposition can be given along the lines of the proof of \cite[Corollary 4.1.2]{PhD}.
         \end{proof}
         \begin{remark}
         For the purpose of proving Theorem \ref{thm:genRiesz} it is enough to use a weaker version of Proposition \ref{cor:genMarHinfdimindep}, with $\pst$ replaced by $\pi|1/2-1/p|.$ The proof of the weaker statement employs Cowling's \cite[Theorem 3]{Hanonsemi} instead of \cite[Theorem 1]{Carb-Drag}.
         \end{remark}

        We proceed to the proof of Theorem \ref{thm:genRiesz}. Recall that various operators $m(L)$ built on $L$ are defined by the multivariate spectral theorem via \eqref{m(L)def}. In particular, each $m(L)$ is defined on the domain given by \eqref{m(L)dom}; for example, the domain of the operator $L_1+\ldots +L_d$ is the subspace
         $\{f\in L^2(X,\nu)\colon \int_{[0,\infty)^d}|\la_1+\ldots+\la_d|^2 dE_{f,f}(\la)\}.$
         All the formalities that are not properly explained in the proof of Theorem \ref{thm:genRiesz} can be easily derived from the multivariate spectral theorem, see e.g.\ \cite[Theorem 4.16]{schmu:dgen}. We decided not to write them explicitly in the proof in order not to obscure its idea.
        \begin{proof}[Proof of Theorem \ref{thm:genRiesz}]
        Clearly, by the multivariate spectral theorem, $L_r^{\sigma}(\sum_{r=1}^d L_r)^{-\sigma}$ is bounded on $L^2(X,\nu).$

        Denote $L_{(r)}=\sum_{s\neq r} L_s,$ so that $\sum_{r=1}^d L_r=L_r+L_{(r)}.$ Then $(L_r,L_{(r)})$ is a pair of strongly commuting self-adjoint non-negative operators on $L^2(X,\nu),$ which satisfy both the assumptions \eqref{contra} and \eqref{noatomatzero}. Indeed, the $L^p$ contractivity property \eqref{contra} of $\exp(-tL_{(r)})$ follows from the identity $e^{-tL_{(r)}}=\prod_{s\neq r}e^{-t L_{s}},$ $t>0;$ while \eqref{noatomatzero} is immediate once we note that $0\leq E_{L_{(r)}}(\{0\})\leq E_{L_s}(\{0\})=0,$ $s=1,\ldots,d,$ $s\neq r.$

         Defining $m_{\sigma}(z_1,z_2)=z_1^{\sigma}(z_1+z_2)^{-\sigma}$ (here we consider the principal branch of the complex power function) it is not hard to see that $m_{\sigma}$ is holomorphic in ${\bf S}_{\pi/2,\pi/2}$ and uniformly bounded in every ${\bf S}_{\phi},$ $\phi=(\phi_1,\phi_2)\in(0,\pi/2)^2.$ Moreover, we clearly have
         $m_{\sigma}(L_r,L_{(r)})=L_r^{\sigma}(\sum_{r=1}^d L_r)^{-\sigma}.$ Thus, using \eqref{eq:dimindepineq} to the system $(T,S)=(L_r,L_{(r)})$ we obtain the desired inequality \eqref{eq:thm:genRiesz}.
        \end{proof}

        From now on we consider systems of operators $L=(L_1,\ldots,L_d)$ such that each $L_r$ is non-negative and self-adjoint on some $L^2(X_r,\nu_r),$ where $(X_r,\nu_r),$ $r=1,\ldots,d,$ is a $\sigma$-finite measure space. We assume that each $L_r,$ $r=1,\ldots,d,$ satisfies the contractivity condition \eqref{contra} (with respect to $L^p(X_r,\nu_r)$) and the atomlessness condition \eqref{noatomatzero}. Denote $X=X_1\times\cdots\times X_d,$ $\nu=\nu_1\otimes\cdots\otimes \nu_d$ and, for $1\leq p\leq \infty,$ $L^p=L^p(X,\nu),$ $\|\cdot\|_p=\|\cdot\|_{L^p}$ and $\|\cdot\|_{p\to p}=\|\cdot\|_{L^p\to L^p}.$
        Recalling the discussion in Section \ref{sec:set}, we know that the operators $L_r\otimes I_{(r)},$ $r=1,\ldots,d,$ can be also regarded as non-negative self-adjoint strongly commuting operators on $L^2,$ that satisfy \eqref{contra} (with respect to $L^p$) and \eqref{noatomatzero}. In what follows, slightly abusing the notation, we usually do not distinguish between $L_r$ and $L_r\otimes I_{(r)}.$

      Let us see how Theorem \ref{thm:genRiesz} formally implies dimension free bounds of certain Riesz transforms. A multi-dimensional Riesz transform of order $j_r\in\mathbb{N}$ associated to the system $L=(L_1,\ldots,L_d)$ is an operator of the form $\delta_r^{j_r}(\sum_{r=1}^d L_r)^{-j_r/2}.$ Here $\delta_r$ is a certain operator acting on a dense subspace of $L^2(X_r,\nu_r)$ (hence, by a tensorization argument also on a dense subspace of $L^2$). The following corollary of Theorem \ref{thm:genRiesz} is rather informal, though, as we shall soon see, it can be easily formalized in many concrete cases.

     \begin{cor}
     \label{cor:Scheme}
     Let $r=1,\ldots,d,$ be fixed. Assume that, for some $j_r \in \mathbb{N},$ the one-dimensional Riesz transform $\delta_r^{j_r} L_r^{-j_r/2}$ of order $j_r$ is bounded on $L^p(X_r,\nu_r).$ Then the multi-dimensional Riesz transform of order $j_r$ is bounded on $L^p$ and
      \begin{equation}
      \label{eq:corschembound}
      \|\delta_r^{j_r}(L_1+\cdots+L_d)^{-j_r/2}\|_{p\to p}\leq C_{p,j_r}\|\delta_r^{j_r}L_r^{-j_r/2}\|_{L^p(X_r,\nu_r)\to L^p(X_r,\nu_r)},\end{equation}
      where the constant $C_{p,j_r}$ is independent of ('the dimension') $d.$
       \end{cor}
     \begin{proof}
     We decompose
      \begin{equation*}
      \delta_r^{j_r}(L_1+\cdots+L_d)^{-j_r/2}=(\delta_r^{j_r}L_r^{-j_r/2})(L_r^{j_r/2}(L_1+\cdots+L_d)^{-j_r/2}).
      \end{equation*}
      Since the system $L$ satisfies the assumptions of Section \ref{sec:set}, using Theorem \ref{thm:genRiesz} with $\sigma=j_r/2$ together with the fact that $\|\delta_r^{j_r} L_r^{-j_r/2}\|_{L^p(X_r,\nu_r)\to L^p(X_r,\nu_r)}=\|\delta_r^{j_r} L_r^{-j_r/2}\|_{p \to p}$ (cf.\ \eqref{eq:tensnormeq}), we obtain the desired bound \eqref{eq:corschembound}.
     \end{proof}
\begin{remark1}
In some cases we need a variant of the corollary that allows the operators $L_r$ to violate the atomlessness condition \eqref{noatomatzero}. Then, we need to add appropriate projections in the definitions of Riesz transforms. More details are provided in the specific cases when we use such a variant.
\end{remark1}
\begin{remark2}
In what follows, see Theorem \ref{thm:RieszOrt} and Theorem \ref{thm:Rieszdiscgen}, we are mostly interested in applying the corollary to first order multi-dimensional Riesz transforms, i.e.\ with $j_r=1.$
\end{remark2}
\begin{remark3}
The argument used in the proof of the corollary bears a resemblance with the method of rotations by Calder\'on and Zygmund, see \cite{Cal-Zyg1} or \cite[Corollary 4.8]{Duo}. Indeed, when applied to the classical (multi-dimensional) Riesz transforms on $\mathbb{R}^d,$ this method allows us to deduce their $L^p$ boundedness from the $L^p$ boundedness of the (one-dimensional) directional Hilbert transforms. However, the method of rotations does not give a dimension free bound for the classical Riesz transforms.
\end{remark3}

Till the end of the paper we focus on rigorous applications in particular cases of Corollary \ref{cor:Scheme} or its variations.

First observe that Corollary \ref{cor:Scheme} implies a dimension free estimate for the norms on $L^p(\mathbb{R}^d,dx),$ $1<p<\infty,$ of the classical Riesz transforms $R_r,$ $r=1,\ldots,d,$ cf.\ \cite{SteinRiesz}. In this case $L_r$ coincides with the self-adjoint extension on $L^2(\mathbb{R}^d,dx)$ of the operators $-\partial_r^2,$ $r=1,\ldots,d,$ initially defined on $C_c^{\infty}(\mathbb{R}^d).$ Then $L=(L_1,\ldots,L_d)$ is a system of strongly commuting operators on $L^2(\mathbb{R}^d,dx),$ satisfying the assumptions \eqref{contra} and \eqref{noatomatzero}. The one-dimensional Riesz transforms $\delta_r L_r^{-1/2}=\partial_r(-\partial_r^2)^{-1/2}$ of order $1$ and the operators $L_r^{1/2}(\sum_{r=1}^d L_r)^{-1/2},$ $r=1,\ldots,d,$  are defined as the Fourier multipliers $$\mF(\delta_r L_r^{-1/2} f)(\xi)={\rm sgn}\, \xi_r\,\mF f(\xi)=\frac{\xi_r}{|\xi_r|}\mF f(\xi),\qquad f\in L^2(\mathbb{R}^d,dx)$$ and $$\mF(L_r^{1/2}(L_1+\cdots+ L_d)^{-1/2}f)(\xi)=\frac{|\xi_r|}{|\xi|}\mF f(\xi),\qquad f\in L^2(\mathbb{R}^d,dx),$$ respectively. Since the multi-dimensional Riesz transform $R_r$ is given by the Fourier multiplier $$\mF(R_r f)(\xi)=\frac{\xi_r}{|\xi|}\mF f,\qquad f\in L^2,$$
we see that, for $r=1,\ldots,d,$
$$R_r f=(\delta_r L_r^{-1/2}) (L_r^{1/2}(L_1+\cdots+ L_d)^{-1/2})f,\qquad f\in L^2(\mathbb{R}^d,dx). $$
In order to obtain the desired dimension free bounds, it suffices to note that $\delta_r L_r^{-1/2}$ is a constant times the Hilbert transform in the $x_r$ variable and to apply \eqref{eq:corschembound} with $j_r=1$.

\section[Riesz transforms for classical orthogonal expansions]{Riesz transforms for classical orthogonal expansions}
\label{sec:ort}
\numberwithin{equation}{section}
In this section Corollary \ref{cor:Scheme} is formalized for Riesz transforms connected with various, discrete or continuous, orthogonal expansions.

In order not to introduce a separate notation for each case of the orthogonal expansions, for precise definitions we kindly refer the reader to consult the list of examples in \cite[Section 7]{NSRiesz} (cases a)-h)) and \cite{Bet1} (case i)). In Theorem \ref{thm:RieszOrt} we consider first order multi-dimensional Riesz transforms $R_r$ in one of the following settings::
\begin{enumerate}[a)]
\item Hermite polynomial expansions $\{H_{k}\}_{k\in \mathbb{N}_0^d}$, see \cite[Section 7.1]{NSRiesz};
\item Laguerre polynomial expansions $\{L_{k}^{\alpha}\}_{k\in \mathbb{N}_0^d},$ $\alpha\in(-1,\infty)^d$, see \cite[Section 7.2]{NSRiesz};
\item Jacobi polynomial expansions $\{P_{k}^{\alpha,\beta}\}_{k\in \mathbb{N}_0^d},$ $\alpha,\beta\in(-1,\infty)^d$, see \cite[Section 7.3]{NSRiesz};
\item Hermite function expansions $\{h_{k}\}_{k\in \mathbb{N}_0^d}$, see \cite[Section 7.4]{NSRiesz};
\item Laguerre expansions of Hermite type $\{\varphi_{k}^{\alpha}\}_{k\in \mathbb{N}_0^d},$ $\alpha\in(-1,\infty)^d$, see \cite[Section 7.5]{NSRiesz};
\item Laguerre expansions of convolution type $\{\ell_{k}^{\alpha}\}_{k\in \mathbb{N}_0^d},$ $\alpha\in(-1,\infty)^d$, see \cite[Section 7.6]{NSRiesz};
\item Jacobi function expansions $\{\phi_{k}^{\alpha,\beta}\}_{k\in \mathbb{N}_0^d},$ $\alpha,\beta \in(-1,\infty)^d$, see \cite[Section 7.7]{NSRiesz};
\item Fourier-Bessel expansions $\{\psi_{k}^{v}\}_{k\in \mathbb{N}^d},$ $v\in(-1,\infty)^d,$ with the Lebesgue measure, see \cite[Section 7.8]{NSRiesz} (there $\nu$ is used in place of $v$);
\item Bessel expansions (multi-dimensional Hankel transform), with the parameter $\alpha\in (-1/2,\infty)^d,$ see \cite[pp. 946-947]{Bet1} (there $\lambda$ is used instead of $\alpha$).
\end{enumerate}
As showed in \cite{NSRiesz}, in the cases a)-h) each of the operators $R_r,$ $r=1,\ldots,d,$ is bounded on $L^2(X,\nu)$ with an appropriate $(X,\nu)$ (in the terminology of \cite{NSRiesz} $\mu$ is used instead of $\nu$). By \cite[Theorem 1.3]{Bet1} the same is true in case i), with $(X,\nu)=((0,\infty)^d,\otimes_{r=1}^d x_r^{2\alpha_r}),$ $\alpha_r>-1/2,$ $r=1,\ldots,d.$

Moreover, in all the settings a)-i) the operators $R_r,$ $r=1,\ldots,d,$ are in fact bounded on $L^p:=L^p(X,\nu),$ $1<p<\infty.$ Additionally, in the settings a) - e), it is known that their norms as operators on $L^p$ are independent of the dimension $d,$ and of their respective parameters (appropriately restricted in some cases). In the setting a) this is a consequence of \cite{Mey1}, see also \cite{Gut1} and \cite{Pis1}. For the dimension free bounds in the settings b) and c), see \cite{No1}, and \cite{NSj2}, respectively. In the setting d) the dimension free boundedness is proved in \cite{HRST}, while in the setting e) in \cite{StWr}.

Theorem \ref{thm:RieszOrt} below is a fairly general result giving dimension free boundedness for the norms on $L^p$ of first order Riesz transforms connected with the orthogonal expansions listed in items a)-i)\footnote{Recently Forzani et al.\ obtained dimension free $L^p$ estimates for Riesz transforms connected with polynomial expansions, see the presentation \cite{SasForSc}. Their results overlap with our Theorem \ref{thm:RieszOrt} in the settings a)-c).}. In fact the theorem is a tool for reducing the problem of proving dimensionless and parameterless bounds to proving parameterless bounds in one dimension. Note that, in the cases f) - i), it seems that the obtained results are new. Additionally, in the cases b) and e), we improve the range of the admitted parameters.

In the statement of Theorem \ref{thm:RieszOrt} by $\diamond_{r}$ we denote the $r$-th parameter in one of the cases b)-c), e)-i). In case i) we write $\alpha$ instead of $\la$ that was used in \cite{Bet1}. By convention, in the cases a) and d), the symbol $C(p,\diamond_r)$ denotes a constant depending only on $p.$ Recall the abbreviation $\|\cdot\|_{p\to p}=\|\cdot\|_{L^p\to L^p}.$
\begin{thm}
\label{thm:RieszOrt} Fix $r\in\mathbb{N},$ and let $R_r$ be the $r$-th first order $d$-dimensional, $d \geq r,$ Riesz transform in one of the settings a)-i), as defined in \cite[Sections 7.1-7.8]{NSRiesz} (items a)-h)) or \cite[p.946]{Bet1} (item i)). Then, $R_r$ is bounded on $L^p$ and
\begin{equation}
\label{eq:RieszOrtbound}
\|R_r\|_{p\to p}\leq C(p,\diamond_r),\qquad 1<p<\infty,
\end{equation}
where the constant $C(p,\diamond_r)$ is independent of both the dimension $d$ and all the other parameters $\diamond_{r'},$ $r'\neq r.$ The bound \eqref{eq:RieszOrtbound} in the following settings below holds for:
\begin{enumerate}[a)]
\setcounter{enumi}{1}
\item parameter $\alpha\in(-1,\infty)^d,$ the constant $C(p,\alpha_r)$ is independent of $\alpha_r;$
\item parameters $\alpha,\beta\in[-1/2,\infty)^d,$ the constant $C(p,(\alpha_r,\beta_r))$ is independent of $\alpha_r,\beta_r;$
\setcounter{enumi}{4}
\item parameter $\alpha\in (\{-1/2\}\cup [1/2,\infty))^d,$ the constant $C(p,\alpha_r)$ is independent of $\alpha_r;$
\item parameter $\alpha\in (-1,\infty)^d;$
\item parameters $\alpha,\beta\in(\{-1/2\}\cup [1/2,\infty))^d;$
\item parameter $\nu\in (\{-1/2\}\cup [1/2,\infty))^d;$
\item parameter $\alpha \in [0,\infty)^d,$ the constant $C(p,\alpha_r)$ is independent of $\alpha_r.$
\end{enumerate}
\end{thm}
\begin{remark}
Analogous (dimension free) estimates are also true for specific higher order Riesz transforms $\delta_r^{j_r}(\sum_{r=1}^d L_r)^{-1/2}$ in the settings a)-i). The proof is similar to that of Theorem \ref{sec:ort} and uses Corollary \ref{cor:Scheme} or its akin. Note that, since $\delta_r$ may not commute with $L_r$ we can not allow more general Riesz transforms in Corollary \ref{cor:Scheme}. Indeed, in all the cases a)-i), the operator $\delta_r$ does not commute with $L_r$.
\end{remark}
\begin{proof}
We focus on proving two particular model cases of the theorem, namely a) and i). The proofs in the other settings are similar to the proof in case a). Whenever we are in one of the settings a)-i), the symbol $L=(L_1,\ldots,L_d)$ denotes the system of operators considered in this setting. Note that each $L_r,$ $r=1,\ldots,d,$ is self-adjoint and non-negative on $L^2(X_r,\nu_r).$ It may happen however that some of the operators $L_r$ do not satisfy one or both of the conditions \eqref{contra} or \eqref{noatomatzero}. Throughout the proof we do not distinguish between $L_r$ and $L_r\otimes I_{(r)}=I\otimes\cdots \otimes L_r\otimes \cdots \otimes I.$

Since the operators $L_r,$ $r=1,\ldots,d,$ act on separate variables, they (or rather their tensor product versions) commute strongly. Hence, by the multivariate spectral theorem, we can define the projections $\Pi_{0}=\chi_{\{\la_1+\cdots+\la_d>0\}}(L)$ and $\Pi_{0,r}=\chi_{\{\la_r>0\}}(L),$ $r=1,\ldots,d,$ via \eqref{m(L)def}. Note that in some of the cases a)-i) or for some values of the parameters the operators $\Pi_0$ or $\Pi_{0,r}$ are trivial, i.e.\ equal to the identity operator.

First we prove formally case a) of Theorem \ref{thm:RieszOrt}. The proof is a minor variation on Corollary \ref{cor:Scheme}. The notation used here is the one introduced in \cite[p. 689]{NSRiesz}.

In particular $L_r=\mL_{r}$ are the one-dimensional Ornstein-Uhlenbeck operators $$\mL_r=-\partial_r^2+2 x_r \partial_r,\qquad r=1,\ldots,d.$$ These operators are symmetric on $C_c^{\infty}(\mathbb{R})$ with respect to the inner product on $L^2(X_r,\nu_r),$ where $(X_r,\nu_r)=(\mathbb{R},\gamma_r\,dx_r)$ with $\gamma_r(x_r)=\exp(-x_r^2).$ Moreover, the one-dimensional $L^2(X_r,\nu_r)$ normalized Hermite polynomials $\{\bnH_{k_r}\}_{k_r\in \mathbb{N}_0}$ are eigenfunctions of $L_r$ corresponding to the eigenvalues $2k_r,$ i.e.\ $L_r\bnH_{k_r}=2k_r\bnH_{k_r},$ $k_r\in\mathbb{N}_0.$ Then, for each $r=1,\ldots,d,$ a standard procedure allows us to consider a self-adjoint extension of $L_r$ (still denoted by the same symbol) defined by
$$L_rf=\sum_{k_r=0}^{\infty}2k_r\langle f, \bnH_{k_r}\rangle_{L^2(X_r,\nu_r)} \bnH_{k_r}$$
on the natural domain
$$\Dom(L_r)=\bigg\{f\in L^2(X_r,\nu_r)\colon \sum_{k_r=0}^{\infty}k_r^2\,|\langle f, \bnH_{k_r}\rangle_{L^2(X_r,\nu_r)}|^2<\infty\bigg\}.$$

 Let $\bnH_k=\bnH_{k_1}\otimes\cdots\otimes \bnH_{k_d},$ $k\in \mathbb{N}_0^d,$ be the $L^2$ normalized multi-dimensional Hermite polynomials. In case a), for $m\colon(2\mathbb{N}_0)^d\to \mathbb{C},$ the joint spectral multiplier provided by \eqref{m(L)def} is
\begin{equation}
\label{m(L)defHerm}
m(L_1,\ldots,L_d)f=\sum_{k\in\mathbb{N}^d_0}m(2k_1,\ldots,2k_d)\langle f, \bnH_k\rangle_{L^2} \bnH_k,
\end{equation}
with domain
$$\Dom(m(L_1,\ldots,L_d))=\bigg\{f\in L^2\colon \sum_{k\in \mathbb{N}^d_0}|m(2k_1,\ldots,2k_d)|^2\,|\langle f, \bnH_{k}\rangle_{L^2}|^2<\infty\bigg\}.$$

The operator $\delta_r=\partial_r$ from \cite[p. 689]{NSRiesz} acts on $\bnH_k,$ $k\in\mathbb{N}^d_0,$ by $$\delta_r(\bnH_k)=\sqrt{2k_r}\bnH_{k-e_r},$$
 where $e_r$ is the $r$-th coordinate vector and, by convention, $\bnH_{k-e_r}=0,$ if $k_r=0.$ Then the ($d$-dimensional) Riesz transform $R_r$ and the one-dimensional Riesz transform are well defined on finite linear combinations of Hermite polynomials by $R_r=\delta_r (\sum_{r=1}^d L_r)^{-1/2}\Pi_{0}$ and $\delta_r L_r^{-1/2}\Pi_{0,r},$ respectively. Here the operators $(\sum_{r=1}^d L_r)^{-1/2}\Pi_{0}$ and $L_r^{-1/2}\Pi_{0,r}$ are given by \eqref{m(L)defHerm}, with the appropriate functions $m.$

Observe that the Riesz transform $R_r$ can be written as
\begin{equation}
\label{eq:RieszLaPoldecom}
R_r f=(\delta_rL_r^{-1/2}\Pi_{0,r})(L_r^{1/2}(L_1+\cdots+ L_d)^{-1/2}\Pi_{0,r})f,
\end{equation}
whenever $f$ is a finite linear combination of Hermite polynomials. To show \eqref{eq:RieszLaPoldecom}, first note that for such $f,$
\begin{equation*}
R_rf=(\delta_rL_r^{-1/2}\Pi_{0,r})(L_r^{1/2}(L_1+\cdots+ L_d)^{-1/2}\Pi_{0})f+\delta_r(I-\Pi_{0,r})((L_1+\cdots+ L_d)^{-1/2}\Pi_0).
\end{equation*}
Then, since $\delta_r$ vanishes on $\Ran(I-\Pi_{0,r})=\{(I-\Pi_{0,r})g\colon g\in L^2(X,\nu)\}$ and $\Pi_{0,r}\Pi_0=\Pi_{0,r},$ we obtain \eqref{eq:RieszLaPoldecom}.

In case a) the boundedness on $L^p(X_r,\nu_r)$ of the one-dimensional Riesz transform $\delta_rL_r^{-1/2}\Pi_{0,r}$ follows from \cite{Muck2}. Additionally, the operator $\Pi_{0,r}$ is also bounded on all $L^p(X_r,\nu_r),$ $1<p<\infty.$ Thus, since finite combinations of Hermite polynomials are dense in $L^p,$ $1<p<\infty,$ coming back to \eqref{eq:RieszLaPoldecom} we see that it suffices to prove the dimension free $L^p$ boundedness of the operator $(L_r^{1/2}(\sum_{r=1}^d L_r)^{-1/2}\Pi_{0,r}).$ It is known that the operators $L_r,$ $r=1,\ldots,d,$ satisfy \eqref{contra}, see \cite[Proposition 1.1. i)]{funccalOu}. However, we cannot directly apply Theorem \ref{thm:genRiesz}, since none of these operators satisfies \eqref{noatomatzero}.

To overcome this obstacle we consider the auxiliary systems $$L_{\varepsilon}=(L_1+\varepsilon,\ldots,L_d+\varepsilon),\qquad \varepsilon>0.$$ Using Theorem \ref{thm:genRiesz} for the systems $L_{\varepsilon}$ we obtain
\begin{equation}
\label{eq:RieszLaPolaux}
\|(L_r+\varepsilon)^{1/2}(L_1+\cdots+L_d+d\varepsilon)^{-1/2}\|_{p\to p}\leq C_{p}, \qquad 1<p<\infty,
\end{equation}
with a constant $C_p$ independent both of the dimension $d$ and $\varepsilon>0.$ Now, since $L_r$ has a discrete spectrum, the set $\sigma(L_r)\setminus\{0\}$ is separated from $0.$ Consequently, from the multivariate spectral theorem,
\begin{equation}\label{eq:RieszHelimit}L_r^{1/2}(L_1+\cdots +L_d)^{-1/2}\Pi_{0,r}=\lim_{\varepsilon\to 0^+}(L_r+\varepsilon)^{1/2}(L_1+\cdots+L_d+d\varepsilon)^{-1/2}\Pi_{0,r},\end{equation}
in the strong $L^2$ sense. Using the above equality together with \eqref{eq:RieszLaPolaux} and the fact that the operator $\Pi_{0,r}$ is bounded on $L^p(X_r,\nu_r),$ $1<p<\infty,$ we obtain the desired dimension free bound for $L_r^{1/2}(\sum_{r=1}^d L_r)^{-1/2}\Pi_{0,r}.$

We omit the proofs of cases b)-h) as they are all analogous to the proof in case a). In fact in some of these cases the proofs are simpler then in case a). That is because we can apply directly Corollary \ref{cor:Scheme}, without using the auxiliary systems $L_{\varepsilon}.$

Now we proceed to the proof of case i), which is a rather straightforward application of Corollary \ref{cor:Scheme}. The notation we use is the one from \cite{Bet1}, with the only difference being the use of $\alpha\in[0,\infty)^d$ (instead of $\la$) to denote the parameter. The one-dimensional Bessel operators $L_r,$ $r=1,\ldots,d,$ are initially defined on $C_c^{\infty}((0,\infty))$ by $$ L_r=-\partial_r^2-\frac{2\alpha_r}{x_r}\partial_r.$$ Each $L_r$ is non-negative and symmetric on $L^2(X_r,\nu_r),$ with $(X_r,\nu_r)=((0,\infty),x_r^{2\alpha_r}).$ In this case the self-adjoint extensions of $L_r$ can be described by the use of the (multi-dimensional, modified) Hankel transform.

The Hankel transform is defined by
\begin{equation*}
\Hla (f)(x)=\int_{(0,\infty)^d} f(\la) E_{x}(\la)\,d\nu(\la), \qquad f\in L^1,
\end{equation*}
where $$E_{x}(\la)=\prod_{k=1}^d (x_k \la_k)^{-\alpha_k+1/2}J_{\alpha_k-1/2}(x_k \la_k)=\prod_{k=1}^dE_{x_k}(\la_k).$$ Here $J_{\nu}$ is the Bessel function of the first kind of order $\nu,$ see \cite[Chapter 5]{Leb}.
It is well known that $\Hla$ extends to an isometry of $L^2$ which satisfies $\Hla^2=I.$ For a Borel measurable real-valued function $m:(0,\infty)^d\to \mathbb{C}$ the operator $m(L_1,\ldots,L_d)$ is the self-adjoint operator $m(L_1,\ldots,L_d)=\Hla(m\Hla f)$ defined on the domain $$\Dom(m(L_1,\ldots,L_d))=\left\{f\in L^2\colon m\Hla(f)\in L^2\right\}.$$

Let $$\mathcal{A}=\left\{f\in L^2 \colon \Hla(f)\in C_c^{\infty}((0,\infty)^d)\right\}.$$ Then $\mathcal{A}$ is a dense subspace of $L^2,$ which is invariant under the operators $L_r^{\pm 1/2}$ as well as $(\sum_{r=1}^d L_r)^{\pm 1/2}.$ In fact every $f\in \mathcal{A}$ is a restriction to $(0,\infty)^d$ of a Schwartz function on $\mathbb{R}^d,$ which is even in each variable, cf.\ \cite{St1}. Moreover, it can be verified that the formulae defining the multi-dimensional Riesz transform $R_r=\partial_r(\sum_{r=1}^{d} L_r)^{-1/2}$ and the one-dimensional Riesz transform $\partial_r L_r^{-1/2},$ see \cite[Theorem 1.3]{Bet1}, are also valid for $f\in \mathcal{A}.$ Hence, we deduce that
\begin{equation}
\label{eq:Bessdecom} R_r f=(\partial_r L_r^{-1/2})(L_r^{1/2}( L_1+\cdots +L_d)^{-1/2}f),\qquad f\in \mathcal{A}.
\end{equation}
Since both $\partial_r L_r^{-1/2}$ and $L_r^{1/2}( \sum_{r=1}^dL_r)^{-1/2}$ are bounded on $L^2,$ a density argument shows that \eqref{eq:Bessdecom} holds on $L^2.$
In case i) the parameter free bounds of the one-dimensional Riesz transforms $\partial_r L_r^{-1/2},$ $r=1,\ldots,d,$ on $L^p(X_r,\nu_r),$ $1<p<\infty,$ were proven in \cite{Vil1}. Note that here we restrict to $\alpha_r\geq 0.$ Moreover, it is well known that the operators $L_r,$ satisfy \eqref{contra} and \eqref{noatomatzero}. Thus, using Theorem \ref{thm:genRiesz}, case i) is proved.

We finish the proof by pointing out exemplary references, in the remaining cases b)-h), for the boundedness of appropriate one-dimensional Riesz transforms and for the $L^p$ contractivity of appropriate semigroups.

For the first of these topics the references are e.g.:\ \cite[Theorem 3b)]{Muck1} for case b) (with a bound independent of the parameter $\alpha_r\in(-1,\infty)$), \cite{NSj2} for case c) (the bound being independent of the parameters), \cite{ThanRieszHerm} for case d), \cite{NSz1} for case f), \cite[Theorem 1.14, Corollary 17.11]{Muck3} for case g), and \cite{CiSt1} for case h). In case e), a combination of the proof of \cite[Theorem 3.3]{NSRieszLagHerm} from \cite[Sections 5,7]{NSRieszLagHerm} (for small $\alpha_r$), and \cite[Theorem 3.1 and Theorem 5.1]{StWr} (for large $\alpha_r$), gives a bound independent of $\alpha_r\in \{-1/2\}\cup[1/2,\infty).$

For the $L^p$ contractivity property of appropriate semigroups the reader can consult e.g.: \cite{NScontr} for the cases b), e), and f); \cite[Section 2]{NSj1} for the cases c) and g); \cite{HRST} for case d). In case h) the $L^p$ contractivity follows from the Feynman-Kac formula.

In those references in which multi-dimensional expansions are considered, we only need to use the one-dimensional case $d=1.$

\end{proof}

\section[Riesz transforms on products of discrete groups]{Riesz transforms on products of discrete groups with polynomial volume growth}
\label{sec:disc}
We apply Corollary \ref{cor:Scheme} to the context introduced in the title of the present section.

Let $G$ be a discrete group. Assume that there is a finite set $U$ containing the identity and generating $G.$ We impose that $G$ has polynomial volume growth, i.e., there is $\alpha\in \mathbb{N}_0,$ such that
$
|U^n|\leq C n^{\alpha},
$ where $|F|$ denotes the counting measure (Haar measure) of $F\subset G.$

Fix a finitely supported symmetric probability measure $\mu$ on $G,$ such that $\supp \mu$ generates $G.$ Then the operator $$P f(x)=P_{\mu} f(x)=f*\mu(x)=\sum_{y\in G}\mu(x^{-1}y) f(y)=\sum_{y\in G}\mu(y) f(xy),$$
is a contraction on all $l^p(G),$ $1\leq p\leq \infty.$ Since $\mu$ is symmetric, $P$ is self adjoint on $l^2(G)$, thus $L=(I-P)$ is self-adjoint and non-negative on $l^2(G).$ Moreover, $L$ satisfies \eqref{contra}. Indeed, we have
$e^{-tL}=e^{-t}\sum_{n=0}^{\infty}\frac{P^n }{n!},$ so that
$$\|e^{-tL}\|_{l^p(G)\to l^p(G)}\leq e^{-t}\sum_{n=0}^{\infty}\frac{\|P\|^n_{l^p(G)\to l^p(G)}}{n!}\leq 1, \qquad 1\leq p\leq \infty.$$

Additionally, since $\supp \mu$ generates $G,$ it can be shown that, if $Pf=f$ for some $f\in l^2(G),$ then $f$ is a constant. For the sake of completeness we give a short proof of this observation\footnote{We thank Gian Maria Dall'Ara for showing us this argument.}. Since the measure $\mu$ is real-valued it suffices to focus on real-valued $f\in l^2(G).$ Then, either $f$ or $-f$ attains a maximum on $G.$ Assume the former holds (the proof in the latter case is analogous) and let $x\in G$ be such that $\sup_{y\in G}f(y)\leq f(x).$ Since $f(x)=Pf(x)=\sum_{y\in \supp \mu}\mu(y)f(xy),$ we have $f(x)=f(xy),$ for all $y\in \supp \mu.$ Next, an inductive argument gives $f(x)=f(xy^n),$ for all $y\in \supp \mu$ and $n\in \mathbb{N}.$ Thus, using the assumption that $\supp \mu$ generates $G,$ we obtain that $f$ is constant on $G.$

From the previous paragraph it follows that, in the case $|G|=\infty,$ if $Pf=f$ and $f\in l^2,$ then $f=0.$ Hence, if $G$ is infinite, then $L$ satisfies \eqref{noatomatzero}, i.e.\ $E_{L}(\{0\})=0.$

For fixed $g_0\in G$ denote $\partial_{g_0} f(x)=f(xg_0)-f(x).$ Then $\partial_{g_0}$ is a bounded operator on all $l^p(G),$ $1\leq p\leq \infty.$ Let $R$ be the discrete Riesz transform considered in \cite{AlexdisRiesz} (see also \cite{HebSC}). From \cite[Theorem 2.4]{AlexdisRiesz} (or \cite[Section 8]{HebSC}) it follows that $R$ is bounded on all $l^p(G),$ $1<p<\infty$ and of weak type $(1,1).$ Note that, if $G$ is infinite, then $E_{L}(\{0\})=0$ and the formula $R=\partial_{g_0}L^{-1/2}=\partial_{g_0}(I-P)^{-1/2}$ holds on a dense subset of $l^2.$ If $|G|=K,$ for some $K\in\mathbb{N}$, then, for all $f\in l^2(G)$ we have $R=\partial_{g_0}L^{-1/2}\pi_0,$ where $\pi_0$ is the projection onto the orthogonal complement of the constants given by\ $\pi_0f=f-|G|^{-1}\sum_{y\in G}f(y).$

Now we consider multi-dimensional Riesz transforms on direct products $G^d.$ Till the end of this section by $l^p,$ $\|\cdot\|_p,$ and $\|\cdot\|_{p\to p},$ we mean $l^p(G^d),$ $\|\cdot\|_{l^p(G^d)},$ and $\|\cdot\|_{l^p(G^d)\to l^p(G^d)},$ respectively.

Let $P_r=P\otimes I_{(r)}$ be given by \eqref{eq:tensnot}. Set $L_r=L\otimes I_{(r)}=(I-P_r)$ and, in the case of finite $G,$ denote \begin{equation}\label{sec:disc,eq:Pi0def}\Pi_0f=f-\frac{1}{|G|^d}\sum_{y=(y_1,\ldots,y_d)\in G^d}f(y).\end{equation} The $d$-dimensional Riesz transform $R_r$ is then defined by \begin{equation}\label{sec:disc,eq:Rieszmultdef}
R_r=\left\{ \begin{array}{lr}
 (\partial_{g_0}\otimes I_{(r)})(L_1+\cdots +L_d)^{-1/2}, &\mbox{if $|G|=\infty$}, \\
 (\partial_{g_0}\otimes I_{(r)})(L_1+\cdots +L_d)^{-1/2}\Pi_0, &\mbox{if $|G|<\infty.$}
       \end{array} \right.
\end{equation}
Recall that the measure $\mu$ in the definition of $P_r=P\otimes I_{(r)}=P_{\mu}\otimes I_{(r)},$ $r=1,\ldots,d,$ is a symmetric probability measure, such that $\supp \mu$ is finite and generates $G$.

By using Corollary \ref{cor:Scheme} (or its variation) we prove the following.
\begin{thm}
\label{thm:Rieszdiscgen}
Let $G$ be a discrete group of polynomial volume growth. Then the $d$-dimensional Riesz transforms $R_r,$ given by \eqref{sec:disc,eq:Rieszmultdef}, are bounded on all $l^p,$ $1<p<\infty.$ Moreover, for each $r=1,\ldots,d,$
\begin{equation}
\|R_r \|_{p\to p}\leq C_p ,\qquad 1<p<\infty,
\label{sec:disc,eq:Rieszdiscgen}
\end{equation}
where the constant $C_p$ is independent of $d.$ Consequently,
\begin{align}\label{sec:disc,eq:Rieszdiscgenvectp<2}
&\bigg\|\left(\sum_{r=1}^d|R_r f|^2\right)^{1/2}\bigg\|_{p}\leq C_p\, d\,\|f\|_{p},\qquad 1<p<2,\\
&\bigg\|\left(\sum_{r=1}^d|R_r f|^2\right)^{1/2}\bigg\|_{p}\leq C_p\, \sqrt{d}\, \|f\|_{p},\qquad 2<p<\infty. \label{sec:disc,eq:Rieszdiscgenvectp>2}
\end{align}
\end{thm}
\begin{remark}
In \cite{BaRus1} and \cite{Rus1} Badr and Russ studied discrete Riesz transforms on graphs. Applying Corollary \ref{cor:Scheme}, Theorem \ref{thm:Rieszdiscgen} can be generalized to products of the graphs studied in \cite{BaRus1} and \cite{Rus1}.
\end{remark}
\begin{proof}[Proof of Theorem \ref{thm:Rieszdiscgen}]
The proof of \eqref{sec:disc,eq:Rieszdiscgen} is just another formalization of Corollary \ref{cor:Scheme} (or its variants).

We start with showing \eqref{sec:disc,eq:Rieszdiscgen} in the case $|G|=\infty.$ The formula
\begin{equation}
\label{sec:disc,eq:RieszdiscGinf}
R_r f=\big((\partial_{g_0}\otimes I_{(r)})L_r^{-1/2}\big)\big(L_r^{1/2}(L_1+\cdots+ L_d)^{-1/2}\big)f\end{equation}
holds on $\Dom(L_1+\cdots +L_d),$ which is a dense subset of $l^2.$ Then, the boundedness on $l^2$ of $R\otimes I_{(r)}=(\partial_{g_0}\otimes I_{(r)})L_r^{-1/2}$ and $L_r^{1/2}(\sum_{r=1}^d L_r)^{-1/2}$ implies that \eqref{sec:disc,eq:RieszdiscGinf} is true on $l^2.$ Now
it suffices to use the boundedness on $l^p$ of the one-dimensional Riesz transform $R\otimes I_{(r)}$, see \cite[Theorem 2.4]{AlexdisRiesz}, together with Theorem \ref{thm:genRiesz}. Note that $L$ satisfies \eqref{contra} and \eqref{noatomatzero}, hence the same is true for the operators $L_r,$ $r=1,\ldots,d.$

To obtain \eqref{sec:disc,eq:Rieszdiscgen} in the case $|G|<\infty,$ we proceed similarly as in the proof of case a) of Theorem \ref{thm:RieszOrt}. Note that here all the formulae are valid on all of $l^2.$ Denoting $$\Pi_{0,r}f(x)=(\pi_0\otimes I_{(r)})f(x)=f(x)-\frac{1}{|G|}\sum_{y_r\in G}f(x_1,\ldots,x_{r-1},y_r,x_{r+1},\ldots,x_d),$$ we see that $\Pi_{0,r}\Pi_0=\Pi_0\Pi_{0,r}=\Pi_{0,r}.$ Since $(\partial_{g_0}\otimes I_{(r)})(I-\Pi_{0,r})=0,$ we rewrite $R_r$ as
\begin{align*}R_r&=\big((\partial_{g_0}\otimes I_{(r)})L_r^{-1/2}\Pi_{0,r}\big)\big(L_r^{1/2}(L_1+\cdots + L_d)^{-1/2}\Pi_{0,r}\big)\\&=(R\otimes I_{(r)})(L_r^{1/2}(L_1+\cdots + L_d)^{-1/2}\Pi_{0,r}).\end{align*}
Using the boundedness of  $R\otimes I_{(r)}$ on $l^p$ we are left with showing that the operator $L_r^{1/2}(\sum_{r=1}^d L_r)^{-1/2}\Pi_{0,r}$ is bounded on $l^p,$ uniformly in $d.$ This can be done exactly as in in the proof of case a) of Theorem \ref{thm:RieszOrt}. Namely, we apply Theorem \ref{thm:genRiesz} to the auxiliary systems $L_{\varepsilon}=(L_1+\varepsilon,\ldots,L_d+\varepsilon),$ to get a uniform in $\varepsilon>0$ and $d$ bound for the $l^p$ norms of the operators $(L_r+\varepsilon)^{1/2}(\sum_{r=1}^d L_r+d\varepsilon)^{-1/2}.$ Using the identity
$$L_r^{1/2}(L_1+\cdots +L_d)^{-1/2}\Pi_{0,r}=\lim_{\varepsilon\to 0^+}(L_r+\varepsilon)^{1/2}(L_1+\cdots+L_d+d\varepsilon)^{-1/2}\Pi_{0,r},$$
cf.\ \eqref{eq:RieszHelimit},
together with the dimension free $l^p$ boundedness of $\Pi_{0,r},$ we thus obtain the desired dimension free boundedness of $L_r^{1/2}(\sum_{r=1}^d L_r)^{-1/2}\Pi_{0,r}.$

Now we focus on proving \eqref{sec:disc,eq:Rieszdiscgenvectp<2} and \eqref{sec:disc,eq:Rieszdiscgenvectp>2} . The former inequality is a simple consequence of the fact that the $l^2$ norm of $(R_1,\ldots,R_d)$ is smaller than its $l^1$ norm. To prove \eqref{sec:disc,eq:Rieszdiscgenvectp>2} we use Minkowski's integral inequality (first inequality below, note that here we need $p/2>1$) together with \eqref{sec:disc,eq:Rieszdiscgen} (second inequality below), obtaining
\begin{equation*}
\bigg\|\left(\sum_{r=1}^d|R_r f|^2\right)^{1/2}\bigg\|_{p}=\bigg\|\sum_{r=1}^d|R_r f|^2\bigg\|^{1/2}_{p/2}\leq \left(\sum_{r=1}^d \|R_r f\|^2_{p}\right)^{1/2}\leq C_p\, \sqrt{d}\, \|f\|_{p}.
\end{equation*}
\end{proof}

We finish this section by showing how Theorem \ref{thm:Rieszdiscgen} can be used to prove a version of \cite[Theorem 2.8]{Lu_Piqu1} by Lust-Piquard, that applies to all cyclic groups. Till the end of this section we assume that $G$ is a cyclic group, that is, there exists $g_0\in G$ such that $\{g_0,g_0^{-1}\}$ generates $G.$
In this case $G$ is abelian and isomorphic to either $(\mathbb{Z},+)$ or $\mathbb{Z}_K=(\{0,\ldots,K-1\},+_{K}),$ where $+_{K}$ denotes addition modulo $K.$ Note that the results of \cite{Lu_Piqu1} include only the cases $G\cong \mathbb{Z}_3,$ $G\cong \mathbb{Z}_4,$ $G\cong \mathbb{Z},$ whereas the case $G\cong \mathbb{Z}_2$ is studied in \cite{Lu_Piqu2}.

Defining $\mu=\mu_{g_0}=(\delta_{g_0}+\delta_{g_0^{-1}})/2,$ we see that $\mu$ is a symmetric probability measure with $\supp \mu=\{g_0,g_0^{-1}\}$ generating $G.$ Moreover, \begin{equation*} \partial_{g_0}\partial_{g_0}^{*}=\partial_{g_0}^*\partial_{g_0}=2(I-P),\end{equation*} where, as we recall, $Pf=P_{\mu}f=\mu*f.$ Consequently, we have
\begin{equation}\label{sec:disc,eq:connect Gen-Lus}
\frac{1}{\sqrt{2}}R_r=\left\{ \begin{array}{lr}
 (\partial_{g_0}\otimes I_{(r)})\big(\sum_{r=1}^d (\partial_{g_0}\otimes I_{(r)})(\partial_{g_0}\otimes I_{(r)})^{*}\big)^{-1/2}, &\mbox{if $|G|=\infty$}, \\
(\partial_{g_0}\otimes I_{(r)})\big(\sum_{r=1}^d (\partial_{g_0}\otimes I_{(r)})(\partial_{g_0}\otimes I_{(r)})^{*}\big)^{-1/2}\Pi_0, &\mbox{if $|G|<\infty,$}
       \end{array} \right.
\end{equation}
with $\Pi_0$ given by \eqref{sec:disc,eq:Pi0def}. Thus, for the specific choice of $\mu=\mu_{g_0},$ our Riesz transform $R_r$ coincides with $\sqrt{2}$ times the Riesz transform considered in \cite[p.\ 307]{Lu_Piqu1}. Since, the group $G$ clearly has polynomial volume growth, we obtain the following corollary of Theorem \ref{thm:Rieszdiscgen}, which is a generalization of \cite[Theorem 2.8]{Lu_Piqu1} to all cyclic groups.
\begin{cor}
\label{cor:RieszgenLu_Piqu}
Let $G$ be a cyclic group. Then the discrete Riesz transforms $2^{-1/2}R_r$ considered in \cite{Lu_Piqu1} and given by \eqref{sec:disc,eq:connect Gen-Lus}, satisfy \eqref{sec:disc,eq:Rieszdiscgen}, \eqref{sec:disc,eq:Rieszdiscgenvectp<2}, and \eqref{sec:disc,eq:Rieszdiscgenvectp>2}.
\end{cor}
\begin{remark1}
Our method, contrary to the one used in \cite{Lu_Piqu1}, does not prove dimension free estimates for the vector of Riesz transforms. Note that, under the assumptions that $G$ is a locally compact abelian group, such that $g_0$ spans an infinite subgroup, by \cite[Theorem 2.8]{Lu_Piqu1} the inequality \eqref{sec:disc,eq:Rieszdiscgenvectp>2} holds with a constant independent of $d.$ However, the counterexample given in \cite[Proposition 2.9]{Lu_Piqu1}, shows that, even for $G=\mathbb{Z},$ the constant in \eqref{sec:disc,eq:Rieszdiscgenvectp<2} is dependent on $d.$
\end{remark1}
\begin{remark2}
The constant $C_p$ in \eqref{sec:disc,eq:Rieszdiscgen}, \eqref{sec:disc,eq:Rieszdiscgenvectp<2}, and \eqref{sec:disc,eq:Rieszdiscgenvectp>2}, is also independent on the considered cyclic group $G.$ To see this we apply transference methods due to Coifman and Weiss. Namely, from \cite[Corollary 3.16]{CWtr}, it is not hard to deduce that the norms of the one-dimensional Riesz transform on $l^p(\mathbb{Z}_K)$ and $l^p(\mathbb{Z})$ are related by $\|R\|_{l^p(\mathbb{Z}_K) \to l^p(\mathbb{Z}_K)}\leq 2\|R\|_{l^p(\mathbb{Z})\to l^p(\mathbb{Z})}.$ Since every cyclic group $G$ is isomorphic to either $\mathbb{Z}$ or $\mathbb{Z}_K$ we thus have $\|R\|_{l^p(G) \to l^p(G)}\leq 2\|R\|_{l^p(\mathbb{Z})\to l^p(\mathbb{Z})}$ Hence, recalling that $C_p$ is of the form $C'_{p}\|R\|_{l^p(G) \to l^p(G)},$ where $C'_p$ depends only on $p$ and not on $G,$ we obtain the desired $G$ independence of $C_p.$
\end{remark2}


\begin{thebibliography}{16}
\bibitem{AlFrMc} D. Albrecht, E. Franks and A. McIntosh, \textit{Holomorphic functional calculi and sums of commuting operators}, Bull. Aust. Math. Soc. 58 (1998), 291--305.
\bibitem{AlexdisRiesz} G. Alexopoulos, \textit{Random walks on discrete groups of polynomial volume growth}, Ann. Probab. (2) 30 (2002), 723--801.

\bibitem{BaRus1} N. Badr, E. Russ, \textit{Interpolation of Sobolev Spaces, Littlewood-Paley Inequalities and Riesz Transforms on Graphs}, Publ. Mat. (2) 53 (2009), 273--328.
\bibitem{Bakr1} D. Bakry, \textit{Etude des transformations de Riesz dans les variétés riemanniennes \`a courbure de Ricci minorée}, Sémin.\ Probab.\ Strasbourg (1) 21 (1987), 137--172.
\bibitem{Bet1} J.J. Betancor, A.J. Castro, and J. Curbelo, \textit{Harmonic analysis operators associated with multidimensional Bessel operators}, Proc. Roy. Soc. Edinburgh Sect. A (5) 142 (2012), 945--974.

\bibitem{Cal-Zyg1} A.\ P.\ Calder\'on, A.\ Zygmund, \textit{On singular integrals}, Amer.\ J.\ Math.\ (2) 78 (1956), 289--309.
\bibitem{Carb-Drag} A. Carbonaro, O. Dragi\v{c}evi\'{c}, \textit{Functional calculus for generators of symmetric contraction semigroups}, preprint (2013), http://arxiv.org/pdf/1308.1338v1.pdf.
\bibitem{CiSt1}
    \'O. Ciaurri, K. Stempak, \textit{Conjugacy for Fourier-Bessel expansions}, Studia Math. 176 (2006), 215--247.
\bibitem{CWtr} R.\ R.\ Coifman, G.\ Weiss, \textit{Transference Methods in Analysis}, CBMS regional conference series in mathematics, No.\ 31, A.M.S., Providence, R.I., 1977.
\bibitem{CMZ}
T{.} Coulhon, D{.} M\"uller and J{.} Zienkiewicz,
\emph{About Riesz transforms on the Heisenberg groups},
Math{.} Ann{.} 305 (1996), 369--379.
\bibitem{Hanonsemi} M. G. Cowling, \textit{Harmonic analysis on semigroups}, Ann. of Math. 117 (1983), 267--283.

\bibitem{Duo} J. Duoandikoetxea, \textit{Fourier Analysis}, Amer.\ Math.\ Soc.\, Providence, RI, 2001.
\bibitem{DragVol0} O. Dragi\v{c}evi\'{c}, A. Volberg, \textit{Bellman functions and dimensionless estimates of Littlewood-Paley type}, J.\ Operator
Theory (1) 56 (2006), 167--198.
\bibitem{DragVol1} O. Dragi\v{c}evi\'{c}, A. Volberg, \textit{Linear dimension-free estimates for the Hermite-Riesz transforms}, arXiv:0711.2460v2 (2007).
\bibitem{DragVol2} O.\ Dragi\v{c}evi\'{c}, A.\ Volberg, \textit{Linear dimension-free estimates in the embedding theorem for Schr\"odinger operators},  J.\ London Math.\ Soc.\ (1) 85 (2012), 191--222.

%
\bibitem{funccalOu} J. Garc\'ia-Cuerva, G. Mauceri, S. Meda, P. Sj\"{o}gren, and J. L. Torrea, \textit{Functional Calculus for the Ornstein Uhlenbeck Operator}, J. Funct. Anal. 183 (2001), 413--450.
\bibitem{Gut1}
C{.} E{.} Guti\'errez,
\emph{On the Riesz transforms for Gaussian measures},
J{.} Funct{.} Anal{.} 120 (1994), 107--134.
\bibitem{GIT}
C{.} E{.} Guti\'errez, A{.} Incognito, J{.}L{.} Torrea,
\emph{Riesz transforms, $g$-functions and multipliers for the Laguerre semigroup},
Houston J{.} Math{.} 27 (2001), 579--592.

\bibitem{HRST}
E{.} Harboure, L{.} de Rosa, C{.} Segovia and J{.} L{.} Torrea,
\emph{$L^p$-dimension free boundedness for Riesz transforms associated to Hermite functions},
Math{.} Ann{.} 328 (2004), 653--682.
\bibitem{HebSC} W. Hebisch, L. Saloff-Coste, \textit{Gaussian estimates for Markov chains and random walks on groups},
Ann. Probab. (2) 21 (1993), 673--709.

\bibitem{Leb} N.N. Lebedev, {\it Special Functions and Their Applications}, Dover, New York, 1972.
\bibitem{Li1} X-D.\ Li, \textit{Martingale transforms and $L^p$-norm estimates of Riesz
transforms on complete Riemannian manifolds}, Probab.\ Theory Relat.\ Fields 141 (2008), 247--281.
\bibitem{LP2}
F{.} Lust-Piquard,
\emph{Dimension free estimates for Riesz transforms associated to the harmonic oscillator on $\mathbb R^n$},
Pot{.} Anal{.} 24 (2006), 47--62.
\bibitem{Lu_Piqu1} F. Lust-Piquard, \textit{Dimension free estimates for discrete Riesz transforms on products of abelian groups}, Adv. Math. (2) 185 (2004), 289--327.
\bibitem{Lu_Piqu2} F. Lust-Piquard, \textit{Riesz transforms associated with the number operator on the Walsh system and the
fermions}, J. Funct. Anal. 155 (1998), 263--285.


\bibitem{Mey1} P. A. Meyer, \textit{Transformationse de Riesz pour les lois gaussiennes}, S\'eminaire de Proba. XVIII, Springer Lecture Notes 1059 (1984), 179--293.
\bibitem{Muck2} B. Muckenhoupt, \textit{Hermite conjugate expansions}, Trans. Amer. Math. Soc. 139 (1969), 243--260.
\bibitem{Muck1} B. Muckenhoupt \textit{Conjugate functions for Laguerre expansions}, Trans. Amer. Math. Soc. 147 (1970), 403--418.
\bibitem{Muck3} B. Muckenhoupt, \textit{Transplantation theorems and multiplier theorems for Jacobi series}, Mem. Amer. Math. Soc. 356 (1986).

\bibitem{No1} A. Nowak, \textit{On Riesz transforms for Laguerre expansions},
J. Funct. Anal. 215 (2004), 217--240.
\bibitem{NSj1} A. Nowak, P. Sj\"ogren, \textit{Sharp estimates of the Jacobi heat kernel}, preprint (2011) http://arxiv.org/abs/1111.3145.pdf.
\bibitem{NSj2} A. Nowak, P. Sj\"ogren,
\textit{Riesz transforms for Jacobi expansions}, J. Anal. Math. 104 (2008), 341--369.
\bibitem{NSRiesz}
A{.} Nowak, K{.} Stempak,
\emph{$L^2$-theory of Riesz transforms for orthogonal expansions},
J{.} Fourier Anal{.} Appl{.} 12 (2006), 675--711.
\bibitem{NSRieszLagHerm} A. Nowak, K. Stempak, \textit{Riesz transforms and conjugacy for Laguerre function expansions of Hermite type}, J. Funct. Anal. 244 (2007) 399--433.
\bibitem{NScontr}
A{.} Nowak and K{.} Stempak,
\emph{On $L^p$-contractivity of Laguerre semigroups},
Illinois J{.} Math{.} (4) 56 (2012).
\bibitem{NSz1} A. Nowak, T. Z. Szarek, \textit{Calder\'on-Zygmund operators related to Laguerre function expansions of convolution type}, J. Math. Anal. Appl. (2) 388 (2012), 801--816.


\bibitem{Pis1}
G{.} Pisier,
\emph{Riesz transforms: A simpler proof of P.A. Meyer's inequality},
Lecture Notes in Math., (1) 1321 (1988), 485--501.

\bibitem{Rus1} E. Russ, \textit{Riesz transforms on graphs for $1<p<2$}, Math.Scand., 87 (2000), 133--160.

\bibitem{SasForSc} E. Sasso, L. Forzani, and R. Scotto \textit{Dimensional free $L^p$ estimates for Riesz Transforms
associated to Polynomial expansions}, presentation (2013), http://calvino.polito.it/\textasciitilde anfunz/Alba/slides/Sasso.pdf.
\bibitem{schmu:dgen} K. Schm\"udgen, \textit{Unbounded Self-adjoint Operators on Hilbert Space}, Grad. Texts in Math. 265 (2012).
\bibitem{SteinRiesz} E. M. Stein, \textit{Some results in harmonic analysis in $\mathbb R^n$, $n\to \infty$},
Bull. Amer. Math. Soc. 9 (1983), 71--73.
\bibitem{St1} K. Stempak, \textit{A note on Zemanian spaces}, Extracta Math. 12 (1997), 33--40.
\bibitem{StWr} K. Stempak, B. Wr\'obel, \textit{Dimension free $L^p$ estimates for Riesz transforms associated with Laguerre
function expansions of Hermite type}, Taiwanese J. Math. (1) 17 (2013), 63--81.

\bibitem{ThanRieszHerm} S. Thangavelu, \textit{On conjugate Poisson integrals and Riesz transforms for the Hermite expansions}, Colloq. Math. 64 (1993), 103--113.

\bibitem{UrZien} R. Urban, J. Zienkiewicz, \textit{Dimension free estimates for Riesz transforms of some Schrödinger operators}, Israel J. Math. (1) 173 (2009), 157--176.

\bibitem{Vil1}  M. Villani, \textit{Riesz transforms associated to Bessel operators}, Illinois J. Math. (1) 52 (2008), 77--89.

\bibitem{PhD} B. Wr\'obel, \textit{Multivariate spectral multipliers}, PhD Thesis, Scuola Normale Superiore, Pisa and Uniwersytet Wroc\l awski (2014), http://ssdnm.mimuw.edu.pl/pliki/prace-studentow/st/pliki/blazej-wrobel-d.pdf
\end{thebibliography}
\end{document}